\documentclass[11pt,reqno]{amsart}
\usepackage{amsmath}
\usepackage{amssymb}
\usepackage[bookmarks]{hyperref}
\usepackage{enumerate}

\providecommand{\bysame}{\leavevmode\hbox to3em{\hrulefill}\thinspace}
\providecommand{\MR}{\relax\ifhmode\unskip\space\fi MR }
\providecommand{\href}[2]{#2}
\newtheorem{theorem}{Theorem}[section]
\newtheorem{lemma}[theorem]{Lemma}
\newtheorem{corollary}[theorem]{Corollary}
\newtheorem{proposition}[theorem]{Proposition}

\theoremstyle{definition}
\newtheorem{remark}[theorem]{Remark}

\numberwithin{equation}{section}

\allowdisplaybreaks[1]

\begin{document}

\title[Weighted composition operators]{Self-adjoint, unitary, and normal weighted composition operators in several variables}
\author{Trieu Le}
\address{Department of Mathematics and Statistics, Mail Stop 942, University of Toledo, Toledo, OH 43606}
\email{trieu.le2@utoledo.edu}


\begin{abstract} We study weighted composition operators on Hilbert spaces of analytic functions on the unit ball with kernels of the form $(1~-~\langle z,w\rangle)^{-\gamma}$ for $\gamma>0$. We find necessary and sufficient conditions for the adjoint of a weighted composition operator to be a weighted composition operator or the inverse of a weighted composition operator. We then obtain characterizations of self-adjoint and unitary weighted composition operators. Normality of these operators is also investigated.
\end{abstract}
\maketitle

%========================================================================================================
\section{Introduction}\label{S:intro}

Let $\mathbb{B}_n$ denote the open unit ball in $\mathbb{C}^n$. For $\mathcal{H}$ a Banach space of analytic functions on $\mathbb{B}_n$ and $\varphi$ an analytic self-map of $\mathbb{B}_n$, the composition operator $C_{\varphi}$ is defined by $C_{\varphi}h=h\circ\varphi$ for $h$ in $\mathcal{H}$ for which the function $h\circ\varphi$ also belongs to $\mathcal{H}$. Researchers have been interested in studying how the function theoretic behavior of $\varphi$ affects the properties of $C_{\varphi}$ on $\mathcal{H}$ and vice versa. When $\mathcal{H}$ is a classical Hardy space or a weighted Bergman space of the unit disk, it follows from Littlewood Subordination Theorem that $C_{\varphi}$ is bounded on $\mathcal{H}$ (see, for example, \cite[Section 3.1]{CowenCRCP1995}). On the other hand, the situation becomes more complicated in higher dimensions. For $n\geq 2$,  there exist unbounded composition operators on the Hardy and Bergman spaces of $\mathbb{B}_n$, even with polynomial mappings. The interested reader is referred to \cite[Chapter 3]{CowenCRCP1995} for these examples and certain necessary and sufficient conditions for the boundedness and compactness of $C_{\varphi}$.

Let $f:\mathbb{B}_n\to\mathbb{C}$ be an analytic function and let $\varphi$ be as above. The weighted composition operator $W_{f,\varphi}$ is defined by $W_{f,\varphi}h=f\cdot(h\circ\varphi)$ for all $h\in\mathcal{H}$ for which the function $f\cdot(h\circ\varphi)$ also belongs to $\mathcal{H}$. Weighted composition operators have arisen in the work of Forelli \cite{ForelliCJM1964} on isometries of classical Hardy spaces $H^{p}$ and in Cowen's work \cite{CowenTAMS1978,CowenJFA1980} on commutants of analytic Toeplitz operators on the Hardy space $H^2$ of the unit disk. Weighted composition operators have also been used in descriptions of adjoints of composition operators (see \cite{CowenJFA2006} and the references therein). Boundedness and compactness of weighted composition operators on various Hilbert spaces of analytic functions have been studied by many mathematicians (see, for example, \cite{ContrerasIEOT2003,CuckovicJLMS2004,MatacheCAOT2008,UekiSzeged2009} and references therein). Recently researchers have started investigating the relations between weighted composition operators and their adjoints. Cowen and Ko \cite{CowenTAMS2010} and Cowen et al. \cite{CowenPreprint2010} characterize self-adjoint weighted composition operators and study their spectral properties on weighted Hardy spaces on the unit disk whose kernel functions are of the form $K_w(z)=(1-\overline{w}z)^{-\kappa}$ for $\kappa\geq 1$. In \cite{BourdonJMAA2010}, Bourdon and Narayan study normal weighted composition operators on the Hardy space $H^2$. They characterize unitary weighted composition operators and apply their characterization to describe all normal operators $W_{f,\varphi}$ in the case $\varphi$ fixes a point in the unit disk.

The purpose of the current paper is to study self-adjoint, unitary and normal weighted composition operators on a class of Hilbert spaces $\mathcal{H}$ of analytic functions on the unit ball. We characterize $W_{f,\varphi}$ whose adjoint is a weighted composition operator or the inverse of a weighted composition operator. As a consequence, we generalize certain results in \cite{BourdonJMAA2010,CowenTAMS2010,CowenPreprint2010} to higher dimensions and also obtain results that have not been previously known in one dimension.

For any real number $\gamma>0$, let $H_{\gamma}$ denote the Hilbert space of analytic functions on $\mathbb{B}_n$ with reproducing kernel functions $$K^{\gamma}_{z}(w)=K^{\gamma}(w,z)=\frac{1}{(1-\langle w,z\rangle)^{\gamma}}\quad\text{ for } z,w\in\mathbb{B}_n.$$

By definition, $H_{\gamma}$ is the completion of the linear span of $\{K^{\gamma}_z: z\in\mathbb{B}_n\}$ with the inner product $\langle K^{\gamma}_z,K^{\gamma}_w\rangle = K^{\gamma}(w,z)$ (this is indeed an inner product due to the positive definiteness of $K^{\gamma}(w,z)$). It is well known that any function $f\in H_{\gamma}$ is analytic on $\mathbb{B}_n$ and for $z\in\mathbb{B}_n$, we have $f(z)=\langle f,K^{\gamma}_z\rangle$. 

For any multi-index $m=(m_1,\ldots,m_n)\in\mathbb{N}_0^n$ (here $\mathbb{N}_0$ denotes the set of non-negative integers) and $z=(z_1,\ldots,z_n)\in\mathbb{B}_n$, we write $z^m = z_1^{m_1}\cdots z_n^{m_n}$. It turns out that $H_{\gamma}$ has an orthonormal basis consisting of constant multiplies of the monomials $z^{m}$, for $m\in\mathbb{N}_0^n$. The spaces $H_{\gamma}$ belong to the class of weighted Hardy spaces introduced by Cowen and MacCluer in \cite[Section 2.1]{CowenCRCP1995}. They are called (generalized) weighted Bergman spaces by Zhao and Zhu in \cite{ZhaoMSMF2008} because of their similarities with other standard weighted Bergman spaces on the unit ball. In fact, for $\gamma>n$, $H_{\gamma}$ is the weighted Bergman space $A^2_{\gamma-n-1}(\mathbb{B}_n)$, which consists of all analytic functions that are square integrable with respect to the weighted Lebesgue measure $(1-|z|^2)^{\gamma-n-1}dV(z)$, where $dV$ is the Lebesgue volume measure on $\mathbb{B}_n$. If $\gamma=n$, $H_{n}$ is the usual Hardy space on $\mathbb{B}_n$. When $n\geq 2$ and $\gamma=1$, $H_{1}$ is the so-called Drury-Arveson space, which has been given a lot of attention lately in the study of multi-variable operator theory and interpolation (see \cite{AglerAMS2002,Arveson1998} and the references therein). For arbitrary $\gamma>0$, $H_{\gamma}$ coincides with the space $A^2_{\gamma-n-1}(\mathbb{B}_n)$ in \cite{ZhaoMSMF2008} (we warn the reader that when $\gamma<n$, the space $A^2_{\gamma-n-1}(\mathbb{B}_n)$ is not defined as the space of analytic functions that are square integrable with respect to $(1-|z|^2)^{\gamma-n-1}dV(z)$, since the latter contains only the zero function).

\section{Bounded weighted composition operators}\label{S:bounded_WCOs}
As we mentioned in the Introduction, the composition operator $C_{\varphi}$ is not always bounded on $H_{\gamma}$ of the unit ball $\mathbb{B}_n$ when $n\geq 2$. On the other hand, if $\varphi$ is a linear fractional self-map of the unit ball, then it was shown by Cowen and MacCluer \cite{CowenSzeged2000} that $C_{\varphi}$ is bounded on the Hardy space and all weighted Bergman spaces of $\mathbb{B}_n$. It turns out, as we will show below, that for such $\varphi$, $C_{\varphi}$ is always bounded on $H_{\gamma}$ for any $\gamma>0$. We will need the following characterization of $H_{\gamma}$, which follows from \cite[Theorem~13]{ZhaoMSMF2008}. 

For any multi-index $m=(m_1,\ldots,m_n)$ of non-negative integers and any analytic function $h$ on $\mathbb{B}_n$, we write $\partial^{m}h = \frac{\partial^{|m|}h}{\partial z_1^{m_1}\cdots\partial z_n^{m_n}}$, where $|m|=m_1+\cdots+m_n$. For any real number $\alpha$, put $d\mu_{\alpha}(z)=(1-|z|^2)^{-n-1+\alpha}dV(z)$, where $dV$ is the usual Lebesgue measure on the unit ball $\mathbb{B}_n$.

\begin{theorem}\label{T:Zhao-Zhu} Let $\gamma>0$. The following conditions are equivalent for an analytic function $h$ on $\mathbb{B}_n$.
\begin{enumerate}[(a)]
\item $h$ belongs to $H_{\gamma}$.
\item For some non-negative integer $k$ with $2k+\gamma>n$, all the functions $\partial^{m}h$, where $|m|=k$, belong to $L^2(\mathbb{B}_n,d\mu_{\gamma+2k})$.
\item For every non-negative integer $k$ with $2k+\gamma>n$, all the functions $\partial^{m}h$, where $|m|=k$, belong to $L^2(\mathbb{B}_n,d\mu_{\gamma+2k})$.
\end{enumerate}
\end{theorem}

\begin{remark}\label{R:Zhao-Zhu}
Theorem \ref{T:Zhao-Zhu} in particular shows that for any given positive number $s$, the function $h$ belongs to $H_{\gamma}$ if and only if for any multi-index $l$ with $|l|=s$, $\partial^{l}h$ belongs to $H_{\gamma+2s}$. As a consequence, $H_{\gamma_1}\subset H_{\gamma_2}$ whenever $\gamma_1\leq\gamma_2$.
\end{remark}

Recall that the multiplier space ${\rm Mult}(H_{\gamma})$ of $H_{\gamma}$ is the space of all analytic functions $f$ on $\mathbb{B}_n$ for which $fh$ belongs to $H_{\gamma}$ whenever $h$ belongs to $H_{\gamma}$. Since norm convergence in $H_{\gamma}$ implies point-wise convergence on $\mathbb{B}_n$, it follows from the closed graph theorem that $f$ is a multiplier if and only if the multiplication operator $M_f$ is bounded on $H_{\gamma}$. It is well known that ${\rm Mult}(H_{\gamma})$ is contained in $H^{\infty}$, the space of bounded analytic functions on $\mathbb{B}_n$. For $\gamma\geq n$, it holds that ${\rm Mult}(H_{\gamma})=H^{\infty}$. This follows from the fact that for such $\gamma$ the norm on $H_{\gamma}$ comes from an integral. On the other hand, when $n\geq 2$ and $\gamma=1$ (hence $H_{\gamma}$ is the Drury-Arveson space), ${\rm Mult}(H_{\gamma})$ is strictly smaller than $H^{\infty}$ (see \cite[Remark~8.9]{AglerAMS2002} or \cite[Theorem~3.3]{Arveson1998}). However we will show that if $f$ and all of its partial derivatives are bounded on $\mathbb{B}_n$, then $f$ is a multiplier of $H_{\gamma}$ for all $\gamma>0$.

\begin{lemma}\label{L:multiplier}
Let $f$ be a bounded analytic function such that for each multi-index $m$, the function $\partial^{m}f$ is bounded on $\mathbb{B}_n$. Then $f$ belongs to ${\rm Mult}(H_{\gamma})$, and hence the operator $M_{f}$ is bounded on $H_{\gamma}$ for any $\gamma>0$.
\end{lemma}

\begin{proof}
Let $\gamma>0$ be given. Choose a positive integer $k$ such that $\gamma+2k>n$. Let $h$ belong to $H_{\gamma}$. For any multi-index $m$ with $|m|=k$, the derivative $\partial^{m}(fh)$ is a linear combination of products of the form $(\partial^{t}f)(\partial^{s}h)$ for multi-indexes $s,t$ with $s+t=m$. For such $s$ and $t$, $\partial^{s}h$ belongs to $H_{\gamma+2|s|}\subset H_{\gamma+2k}$ (by Remark \ref{R:Zhao-Zhu}) and $\partial^{t}f$, which is bounded by the hypothesis, is a multiplier of $H_{\gamma+2k}$ (since ${\rm Mult}(H_{\gamma+2k})=H^{\infty}$). Thus, $(\partial^{t}f)(\partial^{s}h)$ belongs to $H_{\gamma+2k}$. Therefore, $\partial^{m}(fh)$ belongs to $H_{\gamma+2k}$. By Theorem \ref{T:Zhao-Zhu}, $fh$ is in $H_{\gamma}$. Since $h$ was arbitrary in $H_{\gamma}$, we conclude that $f$ is a multiplier of $H_{\gamma}$.
\end{proof}

An analytic map from $\mathbb{B}_n$ into itself is a linear fractional map \cite{CowenSzeged2000} if there is a linear operator $A$ on $\mathbb{C}^n$, two vectors $B,C$ in $\mathbb{B}_n$ and a complex number $d$ such that $$\varphi(z)=\frac{Az+B}{\langle z,C\rangle +d}\quad\text{ for } z\in{\mathbb{B}_n}.$$
Using Lemma \ref{L:multiplier} together with the aforementioned Cowen-MacCluer's result, we show that for $\varphi$ a linear fractional self-map of the unit ball, the composition operator $C_{\varphi}$ is bounded on $H_{\gamma}$ for all $\gamma>0$. In \cite{JuryJFA2008}, Jury proves that $C_{\varphi}$ is bounded on $H_{\gamma}$ for all $\gamma\geq 1$ by an approach using kernel functions. He also obtains an estimate for the norm of $C_{\varphi}$ but we do not need it here.

\begin{proposition}\label{P:boundedness_COs}
Let $\gamma>0$ be given. Suppose $\varphi$ is a linear fractional map of $\mathbb{B}_n$ into itself, then $C_{\varphi}$ is bounded on $H_{\gamma}$.
\end{proposition}

\begin{proof} 
Since $C_{\varphi}$ is a closed linear operator, to show that $C_{\varphi}$ is bounded on $H_{\gamma}$, it suffices to show that $h\circ\varphi$ belongs to $H_{\gamma}$ whenever $h$ belongs to $H_{\gamma}$. For $\gamma>n$, this follows from \cite[Theorem 15]{CowenSzeged2000}. 

Now consider $\gamma>\max\{0,n-2\}$. Write $\varphi=(\varphi_1,\ldots,\varphi_n)$. For each $j$, we have $\partial_{z_j}(h\circ\varphi) = (\partial_{z_1}h\circ\varphi)(\partial_{z_j}\varphi_1)+\cdots+(\partial_{z_n}h\circ\varphi)(\partial_{z_j}\varphi_{n})$. For $1\leq k\leq n$, since $\partial_{z_k}h$ belongs to $H_{\gamma+2}$ (by Remark \ref{R:Zhao-Zhu}) and $\gamma+2>n$, we see that $\partial_{z_k}h\circ\varphi$ also belongs to $H_{\gamma+2}$. On the other hand, since $\partial_{z_j}\varphi_k$ is analytic in a neighborhood of the closed unit ball, it satisfies the hypothesis of Lemma \ref{L:multiplier}. Therefore by Lemma \ref{L:multiplier}, the product $(\partial_{z_k}h\circ\varphi)(\partial_{z_j}\varphi_k)$ belongs to $H_{\gamma+2}$. Thus, $\partial_{z_j}(h\circ\varphi)$ is in $H_{\gamma+2}$ for all $1\leq j\leq n$. Now Remark \ref{R:Zhao-Zhu} shows that $h\circ\varphi$ belongs to $H_{\gamma}$.

Repeating the above argument, we obtain the conclusion of the proposition for $\gamma>\max\{0,n-4\}$, then $\gamma>\max\{0,n-6\}$, and so on. Therefore the conclusion holds for all $\gamma>0$.
\end{proof}

\begin{remark} \label{R:boundedness_WCOs}
Proposition \ref{P:boundedness_COs} together with Lemma \ref{L:multiplier} shows that if $\varphi$ is a linear fractional self-map of $\mathbb{B}_n$ and $f$ is analytic on an open neighborhood of $\overline{\mathbb{B}}_n$, then the weighted composition operator $W_{f,\varphi}$ is bounded on $H_{\gamma}$ for all $\gamma>0$.
\end{remark}

We close this section with some elementary properties of bounded weighted composition operators. Suppose $W_{f,\varphi}$ is bounded on $H_{\gamma}$ for some $\gamma>0$. Then the action of the adjoint $W_{f,\varphi}^{*}$ on the kernel functions can be computed easily. Indeed, for any $z,w$ in $\mathbb{B}_n$, by the properties of the reproducing kernel functions,
\begin{align*}
(W^*_{f,\varphi}K^{\gamma}_z)(w) & = \langle W^{*}_{f,\varphi}K^{\gamma}_z,K^{\gamma}_w\rangle = \langle K^{\gamma}_z,f\cdot(K^{\gamma}_w\circ\varphi)\rangle\\
& = \overline{f}(z)\overline{K^{\gamma}_{w}(\varphi(z))} = \overline{f}(z)K^{\gamma}_{\varphi(z)}(w).
\end{align*}
This gives the well known formula 
\begin{equation}\label{Eqn:adjoint_WCO}
W^{*}_{f,\varphi}K^{\gamma}_z = \overline{f}(z)K^{\gamma}_{\varphi(z)}.
\end{equation}

It is straight forward that the set of bounded weighted composition operators on any $H_{\gamma}$ is closed under operator multiplication. In fact for analytic functions $f,g$ and analytic self-maps $\varphi,\psi$ of $\mathbb{B}_n$ for which both $W_{f,\varphi}$ and $W_{g,\psi}$ are bounded on some $H_{\gamma}$, we have
\begin{align}\label{Eqn:multiplication_WCOs}
W_{f,\varphi}W_{g,\psi} & = W_{f\cdot g\circ\varphi, \psi\circ\varphi}.
\end{align}

Another elementary fact we would like to mention is that each non-zero weighted composition operator $W_{f,\varphi}$ is determined uniquely by the pair $f$ and $\varphi$. In fact, suppose $W_{f,\varphi}=W_{g,\psi}$ on $H_{\gamma}$ and $f$ is not identically zero. Then since $f=W_{f,\varphi}K^{\gamma}_0$ and $g=W_{g,\psi}K^{\gamma}_0$, we obtain $f=g$. Now for any $h\in H_{\gamma}$, since $f\cdot(h\circ\varphi-h\circ\psi)=0$ and $f$ is not identically zero, we have $h\circ\varphi = h\circ\psi$. Write $\varphi=(\varphi_1,\ldots,\varphi_n)$ and $\psi=(\psi_1,\ldots,\psi_n)$. Choosing $h(z)=z_j$, we conclude that $\varphi_j=\psi_j$ for $j=1,\ldots, n$. Thus, $\varphi=\psi$.

\section{Unitary weighted composition operators}

Unitary weighted composition operators have been used in the study of Toeplitz operators on Hardy and Bergman spaces, see for example \cite[p.~189]{ZhuAMS2007}. In this section we will characterize all unitary weighted composition operators. In fact, we will show that $W_{f,\varphi}$ is unitary on $H_{\gamma}$ if and only if $\varphi$ is an automorphism and $f$ is a constant multiple of a reproducing kernel function associated with $\varphi$. 

For $a\in\mathbb{B}_n$, we define the normalized reproducing kernel $k^{\gamma}_a$ by
\begin{equation*}
k^{\gamma}_a(w) = K^{\gamma}_a(w)/\|K^{\gamma}_a\| = \frac{(1-|a|^2)^{\gamma/2}}{(1-\langle w,a\rangle)^{\gamma}}\quad\text{ for } w\in\mathbb{B}_n.
\end{equation*}
Let $\varphi_a$ be the Moebius automorphism of the ball that interchanges $0$ and $a$. The formulas in \cite[Section 2.2.1]{RudinSpringer1980} show that $\varphi_{a}$ is a linear fractional map of $\mathbb{B}_n$. Put $U_a=W_{k^{\gamma}_a,\varphi_a}$, the weighted composition operator on $H_{\gamma}$ given by $\varphi_a$ and $k^{\gamma}_a$. By Remark \ref{R:boundedness_WCOs}, $U_a$ is a bounded operator. It turns out that $U_a$ is in fact a self-adjoint unitary operator, that is, $U_a^{*}=U_a$ and $U^2_{a}=1$. This fact is well known and it is a consequence of a change of variables when $H_{\gamma}$ is a weighted Bergman space ($\gamma>n$) or the Hardy space ($\gamma=n$). See \cite[Proposition 1.13]{ZhuSpringer2005} for weighted Bergman spaces and \cite[Proposition 4.2]{ZhuSpringer2005} for the Hardy space. On these spaces, one has the relation \cite[p.~189]{ZhuAMS2007} $U_aT_{\eta}U_a = T_{\eta\circ\varphi_a}$, where $T_{\eta}$ denotes the Toeplitz operator with symbol $\eta$.

For other values of $\gamma$, for example, the Drury-Arveson space, the inner product on $H_{\gamma}$ does not come from a measure on $\mathbb{B}_n$ so the approach using integral formulas does not seem to work. Our approach here makes use of the kernel functions and it works for all $\gamma>0$. We in fact show that for each given $\gamma>0$, for each automorphism $\psi$ of $\mathbb{B}_n$, there corresponds a weight function $f$ for which $W_{f,\psi}$ is a unitary operator on $H_{\gamma}$. The function $f$ depends on $\psi$ and the value of $\gamma$.

\begin{proposition}\label{P:unitary_auto}
Let $\psi$ be an automorphism of $\mathbb{B}_n$. Put $a=\psi^{-1}(0)$ and $b=\psi(0)$. Then the weighted composition operator $W_{k^{\gamma}_{a},\psi}$ is a unitary operator on $H_{\gamma}$ and $W_{k^{\gamma}_{a},\psi}^{*}=W_{k^{\gamma}_a,\psi}^{-1}=W_{k^{\gamma}_{b},\psi^{-1}}$.
\end{proposition}

\begin{proof}
We will make use of the identity
\begin{equation}\label{Eqn:inner_product_auto}
1-\langle\psi(z),\psi(w)\rangle = \dfrac{(1-\langle a,a\rangle)(1-\langle z,w\rangle)}{(1-\langle z,a\rangle)(1-\langle a,w\rangle)},
\end{equation}
which holds for all $z,w\in\overline{\mathbb{B}}_n$ (see \cite[Theorem 2.2.5]{RudinSpringer1980}). With $z=w=0$, \eqref{Eqn:inner_product_auto} gives $|b|=|\psi(0)|=|a|$. For any $z\in\mathbb{B}_n$, we have
\begin{align*}
k^{\gamma}_{b}(\psi(z)) & = \dfrac{(1-|b|^2)^{\gamma/2}}{(1-\langle\psi(z),b\rangle)^{\gamma}}= \dfrac{(1-|b|^2)^{\gamma/2}}{(1-\langle\psi(z),\psi(0)\rangle)^{\gamma}}\\
& = \dfrac{(1-|b|^2)^{\gamma/2}\cdot (1-\langle z,a\rangle)^{\gamma}}{(1-|a|^2)^{\gamma}}\quad\text{ (by \eqref{Eqn:inner_product_auto} with $w=0$)}\\
& = \Big(\dfrac{1-|b|^2}{1-|a|^2}\Big)^{\gamma/2}\dfrac{1}{k^{\gamma}_{a}(z)} = \dfrac{1}{k^{\gamma}_{a}(z)}\quad\text{ (since $|b|=|a|$)}.
\end{align*}
We obtain
\begin{equation}\label{Eqn:identity_auto}
k^{\gamma}_{a}(z)\cdot k^{\gamma}_{b}(\psi(z)) = 1 \quad\text{ for all } z\in\mathbb{B}_n.
\end{equation}

By Remark \ref{R:boundedness_WCOs}, the operators $W_{k^{\gamma}_a,\psi}$ and $W_{k^{\gamma}_b,\psi^{-1}}$ are bounded on $H_{\gamma}$. For $h\in H_{\gamma}$, \eqref{Eqn:identity_auto} gives $W_{k^{\gamma}_a,\psi}W_{k^{\gamma}_b,\psi^{-1}}h  =k^{\gamma}_a\cdot(k^{\gamma}_b\circ\psi)\cdot h = h$. Therefore $W_{k^{\gamma}_a,\psi}W_{k^{\gamma}_b,\psi^{-1}}=I$ on $H_{\gamma}$. Similarly, $W_{k^{\gamma}_b,\psi^{-1}}W_{k^{\gamma}_a,\psi}=1$ on $H_{\gamma}$. Hence $W_{k^{\gamma}_a,\psi}$ is an invertible operator with inverse $W_{k^{\gamma}_b,\psi^{-1}}$.

Now let $z$ and $w$ be in $\mathbb{B}_n$. Using \eqref{Eqn:inner_product_auto}, we compute
\begin{align*}
\big(W_{k^{\gamma}_a,\psi}K^{\gamma}_{\psi(z)}\big)(w) & = k^{\gamma}_a(w)K^{\gamma}_{\psi(z)}(\psi(w))\\
& = \dfrac{(1-|a|^2)^{\gamma/2}}{(1-\langle w,a\rangle)^{\gamma}}\dfrac{1}{(1-\langle\psi(w),\psi(z)\rangle)^{\gamma}}\\
& = \dfrac{(1-|a|^2)^{\gamma/2}}{(1-\langle w,a\rangle)^{\gamma}}\dfrac{(1-\langle w,a\rangle)^{\gamma}\ (1-\langle a,z\rangle)^{\gamma}}{(1-|a|^2)^{\gamma}\ (1-\langle w,z\rangle)^{\gamma}} = \dfrac{K^{\gamma}_{z}(w)}{\overline{k^{\gamma}_{a}}(z)}.
\end{align*}
Thus $W_{k^{\gamma}_a,\psi}K^{\gamma}_{\psi(z)} = K^{\gamma}_{z}/\overline{k^{\gamma}_{a}}(z)$. Using this and formula \eqref{Eqn:adjoint_WCO}, we obtain
\begin{align*}
W_{k^{\gamma}_a,\psi}^{*}W_{k^{\gamma}_a,\psi}(K^{\gamma}_{\psi(z)}) & =\frac{1}{\overline{k^{\gamma}_a}(z)}W^{*}_{k^{\gamma}_a,\psi}(K^{\gamma}_{z}) = K^{\gamma}_{\psi(z)}.
\end{align*}
Since $z$ was arbitrary and $\psi$ is surjective,  this implies, by linearity, that $W^{*}_{k^{\gamma}_a,\psi}W_{k^{\gamma}_a,\psi}h=h$ for all $h$ in the span $\mathcal{M}$ of $\{K^{\gamma}_z: z\in\mathbb{B}_n\}$. Since $W_{k^{\gamma}_a,\psi}$ is bounded on $H_{\gamma}$ and $\mathcal{M}$ is dense in $H_{\gamma}$, we conclude that $W_{k^{\gamma}_a,\psi}^{*}W_{k^{\gamma}_a,\psi}=I$ on $H_{\gamma}$. Therefore $W_{k^{\gamma}_a,\psi}$ is an invertible isometry on $H_{\gamma}$, and hence a unitary operator.
\end{proof}

\begin{corollary}\label{C:unitary_involution}
For any $a$ in $\mathbb{B}_n$, the operator $U_a = W_{k^{\gamma}_a,\varphi_a}$ is a self-adjoint unitary operator on $H_{\gamma}$.
\end{corollary}
\begin{proof}
Since $\varphi_a$ is an automorphism of $\mathbb{B}_n$ with $\varphi_{a}^{-1}=\varphi_a$ and $a=\varphi^{-1}_a(0)$, the corollary follows immediately from Proposition \ref{P:unitary_auto}.
\end{proof}

For any linear operator $V$ on $\mathbb{C}^n$ with $\|V\|\leq 1$, put $\psi_{V}(z)=Vz$ for $z\in\mathbb{B}_n$. Then $\psi_{V}$ is an analytic self-map of the unit ball. We denote by $C_{V}$ the composition operator $C_{\psi_{V}}$ on $H_{\gamma}$. Lemma 8.1 in \cite{CowenCRCP1995} shows that $C_{V}$ is bounded on any $H_{\gamma}$ and $C_{V}^{*}=C_{V^{*}}$ (the boundedness of $C_{V}$ also follows from Proposition \ref{P:boundedness_COs}). When $V$ is unitary, we obtain

\begin{corollary}\label{C:unitary_linear}
For any unitary operator $V$ of $\mathbb{C}^n$, the composition operator $C_{V}$ is a unitary operator on $H_{\gamma}$ with adjoint $C^{*}_{V}=C_{V^{*}} = C_{V^{-1}}$.
\end{corollary}
\begin{proof}
The corollary can be proved by using Proposition \ref{P:unitary_auto} together with the fact that $\psi_{V}$ is an automorphism of $\mathbb{B}_n$ with $\psi^{-1}_V = \psi_{V^{-1}}$ and $\psi(0)=0$. It also follows (more easily) from the identities
\begin{equation*}
C_{V^{*}}C_{V} = C_{VV^{*}} = I = C_{V^{*}V} = C_{V}C_{V^{*}}.\qedhere
\end{equation*}
\end{proof}

Now assume that $\varphi, \psi$ are analytic self-maps of the unit ball and $f,g$ are analytic functions such that the weighted composition operators $W_{f,\varphi}$ and $W_{g,\psi}$ are bounded on $H_{\gamma}$. We seek necessary and sufficient conditions for which $W_{f,\varphi}W_{g,\psi}^{*}=I$ on $H_{\gamma}$.

Consider first the case $\varphi(0)=0$. For any $z$ in $\mathbb{B}_n$, by \eqref{Eqn:adjoint_WCO}, we have $W_{g,\psi}^{*}K^{\gamma}_z=\overline{g}(z)K^{\gamma}_{\psi(z)}$, so $W_{f,\varphi}W_{g,\psi}^{*}K^{\gamma}_z = \overline{g}(z)f K^{\gamma}_{\psi(z)}\circ\varphi$. Therefore,
\begin{equation}\label{Eqn:adjoint_Inverse}
\overline{g}(z)f(w)K^{\gamma}_{\psi(z)}(\varphi(w)) = K^{\gamma}_{z}(w) \text{ for } z,w\in\mathbb{B}_n.
\end{equation}

Letting $w=0$ and using the fact that $K^{\gamma}_{\psi(z)}(\varphi(0))=K^{\gamma}_{\psi(z)}(0)=1$ and $K^{\gamma}_{z}(0)=1$ for all $z\in\mathbb{B}_n$, we obtain $\overline{g}(z)f(0)=1$, which gives $g(z)=1/\overline{f(0)}$. Thus, $g$ is a constant function.

Letting $z=0$ in \eqref{Eqn:adjoint_Inverse} gives $(f(0))^{-1}f(w)K^{\gamma}_{\psi(0)}(\varphi(w)) = K^{\gamma}_{0}(w) = 1$, which implies $f(w) = f(0)/K^{\gamma}_{\psi(0)}(\varphi(w))$ for $w\in\mathbb{B}_n$. Substituting this into \eqref{Eqn:adjoint_Inverse}, we obtain $K^{\gamma}_{\psi(z)}(\varphi(w))/K^{\gamma}_{\psi(0)}(\varphi(w))=K^{\gamma}_{z}(w)$. Thus
\begin{equation*}
\dfrac{(1-\langle\varphi(w),\psi(z)\rangle)^{-\gamma}}{(1-\langle\varphi(w),\psi(0)\rangle)^{-\gamma}} = (1-\langle w,z\rangle)^{-\gamma}.
\end{equation*}
This gives (here we need to use the continuity of $\varphi$ and $\psi$ on $\mathbb{B}_n$)
\begin{equation*}
\dfrac{1-\langle\varphi(w),\psi(z)\rangle}{1-\langle\varphi(w),\psi(0)\rangle} = 1-\langle w,z\rangle\quad\text{ for all $z,w\in\mathbb{B}_n$},
\end{equation*}
which implies
\begin{equation}\label{Eqn:adjoint_Inverse_Composition}
\Big\langle\frac{\varphi(w)}{1-\langle\varphi(w),\psi(0)\rangle}, \psi(z)-\psi(0)\Big\rangle  = \langle w,z\rangle.
\end{equation}
By Lemma \ref{L:unitary_on_Cn} below, there is an invertible linear operator $A$ on $\mathbb{C}^n$ such that $\psi(z)=\psi(0)+Az$ and $\varphi(w)=(1-\langle\varphi(w),\psi(0)\rangle)(A^{*})^{-1}w$ for $z,w\in\mathbb{B}_n$. The latter implies
\begin{align*}
\langle\varphi(w),\psi(0)\rangle & = (1-\langle\varphi(w),\psi(0)\rangle)\cdot\langle(A^{*})^{-1}w,\psi(0)\rangle,
\end{align*}
which gives
\begin{align*}
1-\langle\varphi(w),\psi(0)\rangle = \frac{1}{1+\langle (A^{*})^{-1}w,\psi(0)\rangle} =  \frac{1}{1+\langle w,A^{-1}\psi(0)\rangle}.
\end{align*}
Therefore $\varphi$ is a linear fractional map given by the formula $$\varphi(w) = \frac{(A^{*})^{-1}w}{1+\langle w,A^{-1}\psi(0)\rangle} \quad\text{ for all } w\in\mathbb{B}_n.$$ 

It turns out that in order for $\varphi$ and $\psi$ to be self-maps of the unit ball, $\psi(0)$ must be zero. To show this, we will make use of Cowen-MacCluer's results \cite{CowenSzeged2000} on linear fractional maps. By the definition on \cite[p.~369]{CowenSzeged2000}, the adjoint map of $\varphi$ has the formula $\sigma(w)=A^{-1}w-A^{-1}\psi(0)$. Since $\varphi$ is a self-map of the unit ball, \cite[Proposition 11]{CowenSzeged2000} implies that $\sigma$ is also a self-map of the unit ball. On the other hand, it is clear that $\psi\circ\sigma = \sigma\circ\psi = {\rm id}_{\mathbb{B}_n}$, the identity map of $\mathbb{B}_n$. This shows that both $\psi$ and $\sigma$ are automorphisms of $\mathbb{B}_n$.

To finish the proof, we use the description of the automorphism group of the unit ball \cite[Theorem 2.2.5]{RudinSpringer1980}, which in particular says that any automorphism that does not fix the origin must be a linear fractional map with a non-constant denominator. Since the denominator of $\psi$ is a constant, $\psi$ must fix the origin: $\psi(0)=0$. Therefore we obtain $\varphi(w)=(A^{*})^{-1}w$ and $\psi(w)=Aw$ for $w\in\mathbb{B}_n$. But $\varphi$ and $\psi$ map the unit ball into itself, hence $A$ is a unitary operator. Since $(A^{*})^{-1}=A$, we see that $\varphi(w)=Aw=\psi(w)$ for $w\in\mathbb{B}_n$. Furthermore, since $\psi(0)=0$, we have $$f(w)=f(0)/K^{\gamma}_{\psi(0)}(\varphi(w))=f(0)/K^{\gamma}_{0}(\varphi(w))=f(0),$$ which is a constant function. Since $g(w)=1/\overline{f(0)}$, we have $f(w)\overline{g}(w)=1$ for all $w\in\mathbb{B}_n$.

Thus we have shown the `only if' part of the following proposition. The `if' part is much easier and it follows from Corollary \ref{C:unitary_linear}.
\begin{proposition}\label{P:adjoint_inverse_WCOs}
Let $f,g$ be analytic functions on $\mathbb{B}_n$ and let $\varphi,\psi$ be analytic self-maps of $\mathbb{B}_n$ with $\varphi(0)=0$. Then $W_{f,\varphi}W^{*}_{g,\psi}=I$ on $H_{\gamma}$ if and only if $f, g$ are constant functions with $f\overline{g}\equiv 1$ and there is a unitary operator $A$ on $\mathbb{C}^n$ so that $\varphi(w)=\psi(w)=Aw$ for $w\in\mathbb{B}_n$. In this case, $W_{f,\varphi}$ and $W_{g,\psi}$ are constant multiples of a unitary composition operator.
\end{proposition}

The general case (without the assumption $\varphi(0)=0$) now follows from Proposition \ref{P:adjoint_inverse_WCOs} after multiplying both $W_{f,\varphi}$ and $W_{g,\psi}$ by a unitary operator. 

\begin{theorem}\label{T:adjoint_inverseWCOs}
Let $f,g$ be analytic functions on $\mathbb{B}_n$ and let $\varphi,\psi$ be analytic self-maps of $\mathbb{B}_n$. Then $W_{f,\varphi}W^{*}_{g,\psi}=I$ on $H_{\gamma}$ if and only if $\varphi=\psi$, an automorphism of $\mathbb{B}_n$; and there is a constant $\lambda\neq 0$ such that $f=\lambda k^{\gamma}_{a}$ and $g=(1/{\overline{\lambda}})k^{\gamma}_{a}$, where $a=\varphi^{-1}(0)$.  Furthermore, both $W_{f,\varphi}$ and $W_{g,\psi}$ are constant multiples of the unitary operator $W_{k^{\gamma}_{a},\varphi}$. 
\end{theorem}

\begin{proof}
The `if' part follows from Proposition \ref{P:unitary_auto} so we only need to prove the `only if' part. Put $b=\varphi(0)$. Define
\begin{align*}
\tilde{f} = f\cdot k^{\gamma}_{b}\circ\varphi,\ \tilde{\varphi} = \varphi_{b}\circ\varphi\quad\text{ and }\quad \tilde{g}  = g\cdot k^{\gamma}_{b}\circ\psi,\ \tilde{\psi} = \varphi_{b}\circ\psi.
\end{align*}
Then by \eqref{Eqn:multiplication_WCOs}, $W_{\tilde{f},\tilde{\varphi}}=W_{f,\varphi}U_{b}$ and $W_{\tilde{g},\tilde{\psi}}=W_{g,\psi}U_{b}$. Since $U_b$ is a unitary, we have $W_{\tilde{f},\tilde{\varphi}}W^{*}_{\tilde{g},\tilde{\psi}} = W_{f,\varphi}W^{*}_{g,\psi}.$ Therefore the second product is the identity operator if and only if the first product is the identity operator. Since $\tilde{\varphi}(0)=\varphi_{b}(\varphi(0))=\varphi_{b}(b)=0$, by Proposition \ref{P:adjoint_inverse_WCOs}, $W_{\tilde{f},\tilde{\varphi}}W_{\tilde{g},\tilde{\psi}}^{*}=I$ on $H_{\gamma}$ if and only if $\tilde{f}, \tilde{g}$ are constant functions with $\tilde{f}\cdot\overline{\tilde{g}}\equiv 1$ and there exists a unitary operator $A$ on $\mathbb{C}^n$ such that $\tilde{\varphi}(w)=\tilde{\psi}(w)=Aw$ for $w\in\mathbb{B}_n$. The identity $\varphi_b^{-1}=\varphi_b$ now implies $\varphi(z)=\psi(z)=\varphi_{b}(Az)$ for $z\in\mathbb{B}_n$. Thus $\varphi=\psi$ and they equal an automorphism of $\mathbb{B}_n$. Suppose $\tilde{f}\equiv\lambda\neq 0$ and $\tilde{g}\equiv 1/\overline{\lambda}$. By \eqref{Eqn:identity_auto}, we obtain
\begin{equation*}
f = \frac{\tilde{f}}{k^{\gamma}_b\circ\varphi} = \frac{\lambda}{k^{\gamma}_{\varphi(0)}\circ\varphi} = \lambda k^{\gamma}_{\varphi^{-1}(0)} = \lambda k^{\gamma}_a.
\end{equation*}
Similarly, $g=(1/\overline{\lambda})k^{\gamma}_a$. Thus $W_{f,\varphi} = \lambda W_{k^{\gamma}_a,\varphi}$ and $W_{g,\psi}=(1/\overline{\lambda})W_{k^{\gamma}_a,\varphi}$.
\end{proof}

\begin{corollary}\label{C:unitary_WCOs} Let $f$ be an analytic function on $\mathbb{B}_n$ and $\varphi$ be an analytic self-map of $\mathbb{B}_n$ such that the operator $W_{f,\varphi}$ is bounded on $H_{\gamma}$ for some $\gamma>0$. Then TFAE
\begin{enumerate}[(a)]
\item $W_{f,\varphi}$ is a unitary on $H_{\gamma}$.
\item $W_{f,\varphi}$ is a co-isometry on $H_{\gamma}$.
\item $\varphi$ is an automorphism of $\mathbb{B}_n$ and $f=\lambda k^{\gamma}_{\varphi^{-1}(0)}$ for some complex number $\lambda$ with $|\lambda|=1$.
\end{enumerate}
\end{corollary}

\begin{proof} The implication $(a)\Rightarrow (b)$ is trivial. The implications $(b)\Rightarrow (c) \Rightarrow (a)$ follow from Theorem \ref{T:adjoint_inverseWCOs} in the case $g=f$ and $\psi=\varphi$.
\end{proof}

\begin{remark} The equivalence of $(a)$ and $(b)$ in the above corollary is not surprising in one dimension. This follows from the fact that in one dimension most weighted composition operators are injective. In fact if $f$ is not identically zero and $\varphi$ is not a constant function, then $W_{f,\varphi}$ is injective on any $H_{\gamma}$ on the unit disk. In dimensions greater than one, it may happen that the kernel of $W_{f,\varphi}$ is non-trivial even in the case $f$ does not vanish and $\varphi$ is a non-constant map of $\mathbb{B}_n$. Thus, it might be surprising that all co-isometric weighted composition operators are in fact unitary on $H_{\gamma}$. Corollary \ref{C:unitary_WCOs} also shows that any unitary weighted composition operator on $H_{\gamma}$ is of the form a constant (of modulus one) multiplying a unitary operator in Proposition \ref{P:unitary_auto}.
\end{remark}

\begin{remark} The equivalence between $(a)$ and $(c)$ for weighted composition operators on the Hardy space of the unit disk is shown by Bourdon and Narayan in \cite{BourdonJMAA2010} by a different route. They show that if $W_{f,\varphi}$ is unitary, then $\varphi$ must be a univalent inner function, and hence, an automorphism of the unit disk. 
\end{remark}

In \cite{BourdonJMAA2010}, Bourdon and Narayan go on to characterize the spectra of these unitary weighted composition operators. Their spectral characterizations are based on whether the automorphism $\varphi$ is elliptic, hyperbolic or parabolic. While the case of elliptic automorphisms (which fix a point in $\mathbb{B}_n$) can be carried on to higher dimensions, we have not been able to resolve the other two cases. The following spectral description is a consequence of a result in the next section about normal weighted composition operators.

\begin{proposition}\label{P:spectrum_unitary_WCOs} Let $f$ be an analytic function and $\varphi$ an automorphism of $\mathbb{B}_n$ that fixes a point $p\in\mathbb{B}_n$. Suppose $W_{f,\varphi}$ is unitary on $H_{\gamma}$. Then $|f(p)|=1$; all eigenvalues of $\varphi'(p)$ belong to the unit circle; and the spectrum of $W_{f,\varphi}$ is the closure of the set $$\{f(p)\}\cup\{f(p)\cdot\lambda_1\cdots\lambda_s: \lambda_j\in\sigma(\varphi'(p))\text{ for } 1\leq j\leq s \text{ and } s=1,2,\ldots\}.$$
Here $\sigma(\varphi'(p))$ is the set of eigenvalues of the matrix $\varphi'(p)$.
\end{proposition}

\begin{proof} Since $W_{f,\varphi}$ is normal, the description of its spectrum follows from Proposition \ref{P:spectrum_normal_WCOs} in Section 4 below. Since the spectrum of $W_{f,\varphi}$ must be a subset of the unit circle, we conclude that $|f(p)|=1$ and $|\lambda|=1$ for any $\lambda$ in $\sigma(\varphi'(p))$.
\end{proof}

We end this section with a lemma that was used in the proof of Proposition \ref{P:adjoint_inverse_WCOs}. We only need the finite dimensional version but the infinite dimensional case is also interesting in its own right. This result might have appeared in the literature but since we are not aware of an appropriate reference, we provide here a proof.

\begin{lemma}\label{L:unitary_on_Cn} 
Let $\mathcal{M}$ be a Hilbert space with an inner product denoted by $\langle,\rangle$. Suppose $F$ and $G$ are two maps from the unit ball $\mathcal{B}$ of $\mathcal{M}$ into $\mathcal{M}$ such that $\langle F(w),G(z)\rangle = \langle w,z\rangle$ for all $w,z$ in $\mathcal{B}$. Then there is an orthogonal decomposition $\mathcal{M}=\mathcal{M}_1\oplus\mathcal{M}_2\oplus\mathcal{M}_3$; there are bounded linear operators $A, B$ from $\mathcal{M}$ into $\mathcal{M}_1$ with $B^{*}A=1$; and there are (possibly non-linear) maps $F_1:\mathcal{M}\rightarrow\mathcal{M}_2$ and $G_1:\mathcal{M}\rightarrow\mathcal{M}_3$ such that $F(w)=Aw+F_1(w)$ and $G(z)=Bz+G_1(z)$ for all $w,z$ in $\mathcal{B}$.

If $\mathcal{M}$ has finite dimension, then both $\mathcal{M}_2$ and $\mathcal{M}_3$ are $\{0\}$ and hence $F(w)=Aw$ and $G(z)=Bz=(A^{*})^{-1}z$ for $w,z\in\mathcal{B}$. If, in addition, $F$ and $G$ map $\mathcal{B}$ into itself, then $A$ is a unitary operator.

If $F=G$, then $F_1=G_1=0$; $A=B$; and hence $F(z)=G(z)=Az$ for $z\in\mathcal{B}$. Furthermore, $A$ is an isometry on $\mathcal{M}$.
\end{lemma}

\begin{proof}
Let $\mathcal{N}$ be the closure of the linear span of $\{G(z):z\in\mathcal{B}\}$. Then we have $P_{\mathcal{N}}G = G$ (here $P_{\mathcal{N}}$ is the orthogonal projection from $\mathcal{M}$ onto $\mathcal{N}$) and for all $w,z\in\mathcal{B}$,
\begin{align*}
\langle P_{\mathcal{N}}F(w),G(z)\rangle = \langle F(w),G(z)\rangle = \langle w,z\rangle.
\end{align*}
For any $z,w_{1},w_{2}$ in $\mathcal{B}$ and complex numbers $c_1,c_2$ such that $c_1w_{1}+c_2w_{2}$ also belongs to $\mathcal{B}$, we have
\begin{align*}
& \Big\langle P_{\mathcal{N}}F(c_1w_{1}+c_2w_{2}) - c_1P_{\mathcal{N}}F(w_{1})-c_2P_{\mathcal{N}}F(w_{2}),G(z)\Big\rangle\\
&\quad\quad\quad = \langle c_1w_{1}+c_2w_{2},z\rangle - c_1\langle w_{1},z\rangle - c_2\langle w_{2},z\rangle = 0.
\end{align*}
Since the range of $P_{\mathcal{N}}F$ is contained in $\mathcal{N}$ and the linear span of the set $\{G(z): z\in\mathcal{B}\}$ is dense in $\mathcal{N}$, we conclude that $P_{\mathcal{N}}F(c_1w_{1}+c_2w_{2}) = c_1P_{\mathcal{N}}F(w_{1})+c_2P_{\mathcal{N}}F(w_{2})$. From this, it follows that $P_{\mathcal{N}}F$ extends to a linear operator on $\mathcal{M}$. We call this extension $A$ and denote the closure of its range by $\mathcal{M}_1$. So $A$ can be regarded as an operator from $\mathcal{M}$ into $\mathcal{M}_1$. We have $\langle Aw,G(z)\rangle = \langle w,z\rangle$ for all $w,z\in\mathcal{B}$. We claim that $A$ is a closed operator and hence by the Closed Graph Theorem, it is bounded. Suppose $\{w_m\}$ is a sequence in $\mathcal{M}$ such that $w_m\rightarrow 0$ and $Aw_m\rightarrow y$ as $m\rightarrow\infty$. For $z\in\mathcal{B}$,
\begin{align*}
0 & = \lim_{m\rightarrow\infty} \langle w_m,z\rangle = \lim_{m\rightarrow\infty}\langle Aw_m,G(z)\rangle = \langle y,G(z)\rangle.
\end{align*}
Since $y$ belongs to $\mathcal{M}_1\subset\mathcal{N}$ and the linear span of $\{G(z): z\in\mathcal{B}\}$ is dense in $\mathcal{N}$, we conclude that $y=0$. So $A$ is a closed operator.

Now for $w,z\in\mathcal{B}$, $\langle Aw,P_{\mathcal{M}_1}G(z)\rangle  = \langle Aw,G(z)\rangle = \langle w,z\rangle.$ It then follows, by the same argument as before, that $P_{\mathcal{M}_1}G$ extends to a bounded linear operator on $\mathcal{M}$. Call this operator $B$. Then the range of $B$ is contained in $\mathcal{M}_1$ (hence we may regard $B$ as an operator from $\mathcal{M}$ into $\mathcal{M}_1$) and we have $\langle Aw,Bz\rangle = \langle w,z\rangle$ for $w,z\in\mathcal{B}$. As before, $B$ can be shown to be a closed operator, hence it is bounded and we have $B^{*}A=1$.

Put $\mathcal{M}_2 = \mathcal{M}\ominus\mathcal{N}$ and $\mathcal{M}_3=\mathcal{N}\ominus\mathcal{M}_1$. Put $F_1=P_{\mathcal{M}_2}F$ and $G_1=P_{\mathcal{M}_3}G$. We then have, on $\mathcal{B}$,
\begin{align*}
F & = P_{\mathcal{N}}F + P_{\mathcal{M}_2}F = A+F_1,\\
G & = P_{\mathcal{N}}G = P_{\mathcal{N}}P_{\mathcal{M}_1}G + P_{\mathcal{N}}(I-P_{\mathcal{M}_1})G = P_{\mathcal{M}_1}G+P_{\mathcal{M}_3}G = B+G_1.
\end{align*}

If $\mathcal{M}$ is a finite dimensional space, then it follows from $B^{*}A=1$ that both $A$ and $B$ are invertible operators from $\mathcal{M}$ onto $\mathcal{M}_1$. Therefore, $\mathcal{M}_1=\mathcal{M}$, which forces $\mathcal{M}_2=\mathcal{M}_3=\{0\}$. So $F(w)=Aw$ and $G(z)=Bz=(A^{*})^{-1}z$ for $w,z\in\mathcal{B}$. If both $F$ and $G$ maps $\mathcal{B}$ into itself, then $\|A\|\leq 1$ and $\|(A^{*})^{-1}\|\leq 1$. Consequently, both $A$ and $A^{-1}$ are contractive operators on $\mathcal{M}$. This forces $A$ to be unitary.

If $F=G$ then we have $F_1=G_1=0$ and $A=B$. But $B^{*}A=1$, so $A^{*}A=1$ and hence $A$ is an isometry on $\mathcal{M}$.
\end{proof}

\section{Normal weighted composition operators}

Recall that for $V$ a linear operator on $\mathbb{C}^n$ with $\|V\|\leq 1$, we denote by $C_{V}$ the composition operator induced by the analytic self-map $\psi_{V}(z)=Vz$ of $\mathbb{B}_n$. If $V$ is normal, then since $C_{V}C_{V}^{*} = C_{V^{*}V} = C_{VV^{*}}=C_{V}^{*}C_{V}$, the operator $C_{V}$ is normal on $H_{\gamma}$. It turns out that these are all normal composition operators on $H_{\gamma}$ for each $\gamma>0$. The following result is part of \cite[Theorem 8.2]{CowenCRCP1995}.

\begin{proposition}\label{P:normal_COs}
Let $\gamma>0$ and let $\varphi$ be an analytic mapping of $\mathbb{B}_n$ into itself. Then $C_{\varphi}$ is normal on $H_{\gamma}$ if and only if $\varphi(z)=Az$ for some normal linear operator $A$ on $\mathbb{C}^n$ with $\|A\|\leq 1$.
\end{proposition}

The spectrum of a normal composition operator can be determined easily. Let $A$ be a normal linear operator on $\mathbb{C}^n$ with $\|A\|\leq 1$, we will identify the eigenvalues and eigenvectors of $C_{A}$ on $H_{\gamma}$. We will show that $C_{A}$ is diagonalizable and hence its spectrum is the closure of the set of its eigenvalues.

Since $A$ is normal, there is an orthonormal basis $\{u_1,\ldots, u_n\}$ of $\mathbb{C}^n$ which consists of eigenvectors of $A$. Write $Au_j = \lambda_ju_j$, where $\lambda_j$ is the eigenvalue corresponding to $u_j$ for $1\leq j\leq n$ (note that some of these eigenvalues may be the same). Then the spectrum of $A$ is given by $\sigma(A)=\{\lambda_1,\ldots,\lambda_n\}$. Let $\{e_1,\ldots, e_n\}$ be the standard orthonormal basis for $\mathbb{C}^n$ and let $V$ be the unitary operator on $\mathbb{C}^n$ such that $Vu_j=e_j$ for $1\leq j\leq n$. For any $z=(z_1,\ldots,z_n)$ in $\mathbb{C}^n$, we have
\begin{equation}\label{Eqn:diagonal_matrix}
VAV^{*}(z) = (\lambda_1 z_1,\ldots, \lambda_n z_n).
\end{equation}

Recall from the Introduction that for any $\gamma>0$, the set of analytic monomials $\{z^{m} = z_1^{m_1}\cdots z_n^{m_n}: m=(m_1,\ldots,m_n)\in\mathbb{N}_{0}^n\}$ is a complete orthogonal set in $H_{\gamma}$. By \eqref{Eqn:diagonal_matrix}, we have $C_{VAV^{*}}(z^{m}) = \lambda^{m}z^{m}$ for all $m\in\mathbb{N}_0^n$ (here we write $\lambda^m = \lambda_1^{m_1}\cdots\lambda_n^{m_n}$ and use the convention that $0^{0}=1$). Since $C_{VAV^{*}}=C_{V^{*}}C_{A}C_{V}$ and $C_{V}$ is unitary with $C_{V}^{*} = C_{V^{*}}$ (by Corollary \ref{C:unitary_linear}), we conclude that the set $\{C_{V}z^{m}: m\in\mathbb{N}_0^n\}$ is a complete orthogonal set in $H_{\gamma}$ and for each $m\in\mathbb{N}_0^n$, the function $C_{V}z^{m}$ is an eigenfunction for $C_{A}$ with eigenvalue $\lambda^{m}$. Thus the operator $C_{A}$ is diagonalizable in $H_{\gamma}$ and the spectrum $\sigma(C_{A})$ is the closure of the set $\{\lambda^{m}: m\in\mathbb{N}_0^n\}$.

The eigenfunctions $C_{V}(z^m)$ of $C_{A}$ can be described in terms of the eigenvectors of $A$ as follows. 
\begin{align*}
C_{V}(z^m) & = C_{V}(z_1^{m_1}\cdots z_n^{m_n}) = C_{V}(\langle z,e_1\rangle^{m_1}\cdots\langle z,e_n\rangle^{m_n})\\
& = \langle Vz,e_1\rangle^{m_1}\cdots\langle Vz,e_n\rangle^{m_n} = \langle z,V^{*}e_1\rangle^{m_1}\cdots\langle z,V^{*}e_n\rangle^{m_n}\\
& = \langle z,u_1\rangle^{m_1}\cdots\langle z,u_n\rangle^{m_n}.
\end{align*}
We have thus obtained
\begin{proposition}\label{P:spectrum_normal_COs}
Let $A$ be a normal operator on $\mathbb{C}^n$ with $\|A\|\leq 1$. Let $\{u_1,\ldots, u_n\}$ be an orthonormal basis for $\mathbb{C}^n$ consisting of eigenvectors of $A$. Write $Au_j=\lambda_j u_j$ for $1\leq j\leq n$. Then the following statements hold.
\begin{itemize}
\item[(a)] The set $\{f_{m}(z)=\langle z,u_1\rangle^{m_1}\cdots\langle z,u_n\rangle^{m_n}: m=(m_1,\ldots, m_n)\in\mathbb{N}_0^n\}$ is a complete orthogonal set of $H_{\gamma}$.
\item[(b)] Each $f_m$ is an eigenfunction of $C_{A}$ with eigenvalue $\lambda^{m} = \lambda_1^{m_1}\cdots\lambda_n^{m_n}$.
\item[(c)] The spectrum of $C_{A}$ is the closure of the set $\{\lambda^{m}: m\in\mathbb{N}_0^n\}$, where $\lambda=(\lambda_1,\ldots,\lambda_n)$. This set can also be written as $\{1\}\cup\{\alpha_1\cdots\alpha_s: \alpha_j\in\sigma(A)\text{ for } 1\leq j\leq s\text{ and } s=1,2,\ldots\}$.
\end{itemize}
\end{proposition}

In \cite{BourdonJMAA2010}, Bourdon and Narayan study normal weighted composition operators on the Hardy space of the unit disk. They provide two necessary conditions for $W_{f,\varphi}$ to be normal \cite[Lemma 2 and Proposition 3]{BourdonJMAA2010}: (1) either $f\equiv 0$ or $f$ never vanishes, and (2) if $\varphi$ is not a constant function and $f$ is not the zero function, then $\varphi$ is univalent. While condition (1) is still valid in all dimensions with the same proof, condition (2) no longer holds in dimension greater than one, as Proposition \ref{P:normal_COs} shows. On the other hand, we will see that the characterization of normal $W_{f,\varphi}$ on $H_{\gamma}$ remains the same if the map $\varphi$ fixes a point in the unit ball. Our approach here was inspired by that in \cite{BourdonJMAA2010} but the argument has been simplified. Furthermore, our proof works for all $H_{\gamma}$ in any dimension.

\begin{theorem}\label{T:normal_WCOs}
Suppose $\varphi$ is an analytic self-map of $\mathbb{B}_n$ that fixes a point $p$ in $\mathbb{B}_n$. If $W_{f,\varphi}$ is a non-zero normal operator, then there exist a constant $\alpha\neq 0$ and a normal linear operator $A$ on $\mathbb{C}^n$ with $\|A\|\leq 1$ such that 
\begin{equation}\label{Eqn:normal_WCOs}
f = \alpha\frac{k^{\gamma}_p}{k^{\gamma}_p\circ\varphi},\quad\text{ and }\quad \varphi(z) = \varphi_p(A\varphi_p(z)) \text{ for } z\in\mathbb{B}_n.
\end{equation}
Conversely, if $f$ and $\varphi$ satisfy \eqref{Eqn:normal_WCOs}, then $\alpha = f(p)$ and $W_{f,\varphi}$ is unitarily equivalent to the normal operator $f(p) C_{A}$ (in fact, $W_{f,\varphi}=U_{p}\big(f(p)C_{A}\big)U_{p}$) and hence it is normal.
\end{theorem}

\begin{proof}
We assume first $\varphi(0)=0$ and $W_{f,\varphi}$ is a non-zero normal operator. By \eqref{Eqn:adjoint_WCO}, we have $$W^{*}_{f,\varphi}K^{\gamma}_0 = \overline{f(0)}K^{\gamma}_{\varphi(0)} = \overline{f(0)}K^{\gamma}_{0}.$$
This shows that $K^{\gamma}_{0}$ is an eigenvector of $W_{f,\varphi}^{*}$ with eigenvalue $\overline{f(0)}$. Since $W_{f,\varphi}$ is normal, we obtain $W_{f,\varphi}K^{\gamma}_{0} = f(0)K^{\gamma}_0$, which implies $f\cdot K^{\gamma}_{0}\circ\varphi = f(0)K^{\gamma}_{0}$ and hence $f = f(0)$ since $K^{\gamma}_0\equiv 1$.  So $f$ is a constant function (which is non-zero because $W_{f,\varphi}$ is a non-zero operator). This in turns implies that $C_{\varphi}$ is normal on $H_{\gamma}$. By Proposition \ref{P:normal_COs}, there is a normal linear operator on $\mathbb{B}_n$ with $\|A\|\leq 1$ such that $\varphi(z)=Az$ for $z\in\mathbb{B}_n$.

For general $p$, define $\widetilde{f}=(k^{\gamma}_p\circ\varphi\circ\varphi_p)(f\circ\varphi_p)k^{\gamma}_p$ and $\widetilde{\varphi}=\varphi_p\circ\varphi\circ\varphi_p$. By \eqref{Eqn:multiplication_WCOs}, $U_pW_{f,\varphi}U_p = W_{\widetilde{f},\widetilde{\varphi}}$. Since $W_{f,\varphi}$ and $W_{\widetilde{f},\widetilde{\varphi}}$ are unitarily equivalent (recall that $U_p$ is a self-adjoint unitary operator), one is normal if and only if the other is normal. Since $\widetilde{\varphi}(0)=0$, the above argument shows that $W_{\widetilde{f},\widetilde{\varphi}}$ is normal if and only if $\widetilde{f}$ is a constant function, say, $\widetilde{f}\equiv\alpha$ and $\widetilde{\varphi}(z)=Az$ for some normal operator $A$ on $\mathbb{C}^n$ with $\|A\|\leq 1$. Thus we obtain $(k^{\gamma}_p\circ\varphi\circ\varphi_p)(f\circ\varphi_p)k^{\gamma}_p\equiv\alpha$ and $\varphi_p\circ\varphi\circ\varphi_p(z) = Az$. Using the fact that $\varphi_p\circ \varphi_p$ is the identity map on $\mathbb{B}_n$, we get
\begin{equation}\label{Eqn:second_normal_WCOs}
f = \frac{\alpha}{(k^{\gamma}_p\circ\varphi)(k^{\gamma}_p\circ\varphi_p)},\quad\text{ and }\quad \varphi(z) = \varphi_p( A\varphi_p(z))\text{ for } z\in\mathbb{B}_n.
\end{equation}
On the other hand, since $\varphi_p(0)=p=\varphi_p^{-1}(0)$, \eqref{Eqn:identity_auto} gives $k^{\gamma}_p\cdot (k^{\gamma}_p\circ\varphi_{p}) = 1$. Therefore $f$ can be written as $f=\alpha\,\frac{k^{\gamma}_p}{k^{\gamma}_p\circ\varphi}.$ Since $k^{\gamma}_p(\varphi(p))=k^{\gamma}_p(p)$, we see that $f(p)=\alpha$.

Conversely, if $f$ and $\varphi$ satisfy \eqref{Eqn:normal_WCOs}, then they satisfy \eqref{Eqn:second_normal_WCOs} (with $\alpha = f(p)$) and hence $W_{f,\varphi}$ is unitarily equivalent to the normal operator $f(p) C_{A}$. 
\end{proof}

We now use Theorem \ref{T:normal_WCOs} and Proposition \ref{P:spectrum_normal_COs} to discuss the spectra of normal weighted composition operators. Suppose that $f$ and $\varphi$ satisfy \eqref{Eqn:normal_WCOs}. Let $\{u_1,\ldots,u_n\}$ be an orthonormal basis for $\mathbb{C}^n$ consisting of eigenvectors of $A$. Let $\lambda_j$ be the eigenvalue of $A$ corresponding to the eigenvector $u_j$ and put $f_j(z)=\langle z,u_j\rangle$ for $1\leq j\leq n$. For each multi-index $m=(m_1,\ldots,m_n)$ in $\mathbb{N}_0^n$, we write $f_m = f_1^{m_1}\cdots f_n^{m_n}$. From Proposition \ref{P:spectrum_normal_COs} we know that $\{f_m: m\in\mathbb{N}_0^n\}$ is a complete orthogonal set of $H_{\gamma}$ and $C_{A}(f_m) = \lambda^{m} f_m$ for each $m$, where $\lambda=(\lambda_1,\ldots,\lambda_n)$.

For $1\leq j\leq n$, put
\begin{align*}
g_j(z) & = (U_{p}f_j)(z) = k^{\gamma}_{p}(z)f_j(\varphi_p(z)) = k^{\gamma}_{p}(z)\langle\varphi_p(z),u_j\rangle.
\end{align*}
Put $g_m = g_1^{m_1}\cdots g_n^{m_n} = U_{p}(f_m)$ for $m=(m_1,\ldots,m_n)\in\mathbb{N}_0^n$. Since $U_{p}$ is unitary on $H_{\gamma}$, the set $\{g_m: m\in\mathbb{N}_0^n\}$ is a complete orthogonal set of $H_{\gamma}$. Since $W_{f,\varphi} = U_{p}(f(p)C_{A})U_{p}$ by Theorem \ref{T:normal_WCOs} and $U_p=U^{-1}_p$, we conclude that $W_{f,\varphi}g_m = f(p)\lambda^{m}g_m$ for $m\in\mathbb{N}_0^n$. Therefore the spectrum of $W_{f,\varphi}$ is the closure of the set $\{f(p)\lambda^{m}: m\in\mathbb{N}_0^n\}$, which is the same as
\begin{align*}
\{f(p)\}\cup\{f(p)\cdot\alpha_1\cdots\alpha_s: \alpha_j\in\sigma(A)\text{ for } 1\leq j\leq s \text{ and } s=1,2,\ldots\}.
\end{align*}

On the other hand, by the chain rule, we have
\begin{align*}
\varphi'(p)=\varphi_p'(A\varphi_p(0))A\varphi_p'(p) = \varphi_p'(0)A\varphi_p'(p).
\end{align*}
Since $\varphi_p\circ\varphi_p=I_{\mathbb{B}_n}, \varphi_p(0)=p$ and $\varphi_p(p)=0$, the chain rule again gives $\varphi_p'(p)\varphi_p'(0)=\varphi'_p(0)\varphi_p'(p)=I_n$, the identity operator on $\mathbb{C}^n$. Therefore $\varphi'(p)$ and $A$ are similar and hence they have the same set of eigenvalues, counting multiplicities. In particular, $\sigma(A)=\sigma(\varphi'(p))$. We thus obtain the description of the spectrum of $W_{f,\varphi}$ intrinsically in terms of $f$ and $\varphi$.

\begin{proposition}\label{P:spectrum_normal_WCOs}
Let $f$ be a non-zero analytic function and $\varphi$ an analytic self-map of $\mathbb{B}_n$ that fixes a point $p$ on $\mathbb{B}_n$. Suppose $W_{f,\varphi}$ is a normal operator on $H_{\gamma}$. Then the spectrum of $W_{f,\varphi}$ is the closure of the set $$\{f(p)\}\cup\{f(p)\cdot\alpha_1\cdots\alpha_s: \alpha_j\in\sigma(\varphi'(p))\text{ for } 1\leq j\leq s \text{ and } s=1,2,\ldots\}.$$
\end{proposition}

We have characterized normal weighted composition operators induced by analytic self-maps of $\mathbb{B}_n$ that fix a point in $\mathbb{B}_n$. Our approach (conjugating $W_{f,\varphi}$ by a unitary) does not seem to work for $\varphi$ that only has fixed point on the sphere. In the rest of this section, we investigate normal weighted composition operators of a certain type.

In \cite[Section 5]{BourdonJMAA2010}, Bourdon and Narayan note that in one dimension, the function $f$ in the conclusion of Theorem \ref{T:normal_WCOs} is in fact a constant multiple of $K^{\gamma}_{\sigma(0)}$, where $\sigma$ is the adjoint of the linear fractional map $\varphi$. They then go on to find necessary and sufficient conditions for the normality of $W_{f,\varphi}$, where $\varphi$ is a linear fractional map and $f=K^{\gamma}_{\sigma(0)}$. It turns out that in higher dimensions similar results also hold but they are less obvious because of the complicated settings of several variables.

Recall that a linear fractional map $\varphi$ has the form $\varphi(z) = \frac{Az+B}{\langle z,C\rangle +d}$, where $A$ is a linear operator on $\mathbb{C}^n$; $B, C$ are vectors in $\mathbb{C}^n$; and $d$ is a complex number. Given such a map $\varphi$, its adjoint is defined by $$\sigma(z)=\sigma_{\varphi}(z) = \frac{A^{*}z-C}{-\langle z,B\rangle+\overline{d}}.$$
For more details on $\sigma$ and its relation with $\varphi$, see \cite{CowenSzeged2000}.

We begin by a lemma that can be verified by a direct computation, using the formulas of $\varphi, \sigma$ and of the reproducing kernel functions.
\begin{lemma}\label{L:kernel_identity_adjoint} Let $\varphi$ be a linear fractional self-map of $\mathbb{B}_n$ and let $\sigma$ be its adjoint. Then for any point $a$ in $\mathbb{B}_n$, we have
\begin{equation*}
K^{\gamma}_{\varphi(0)}\cdot K^{\gamma}_{a}\circ\sigma = \overline{K}^{\gamma}_{\sigma(0)}(a)K^{\gamma}_{\varphi(a)}\quad\text{ and }\quad K^{\gamma}_{\sigma(0)}\cdot K^{\gamma}_{a}\circ\varphi  = \overline{K}^{\gamma}_{\varphi(0)}(a)K^{\gamma}_{\sigma(a)}.
\end{equation*}
\end{lemma}

By Remark \ref{R:boundedness_WCOs}, both operators $W_{K^{\gamma}_{\varphi(0)},\sigma}$ and $W_{K^{\gamma}_{\sigma(0)},\varphi}$ are bounded on $H_{\gamma}$. Now the first identity in Lemma \ref{L:kernel_identity_adjoint} together with \eqref{Eqn:adjoint_WCO} shows that $$W_{K^{\gamma}_{\varphi(0)},\sigma}K^{\gamma}_a = W_{K^{\gamma}_{\sigma(0)},\varphi}^{*}K^{\gamma}_a\quad \text{ for all } a\in\mathbb{B}_n,$$ which implies that
\begin{equation}\label{Eqn:adjusted_Cowen_MacCluer}
W_{K^{\gamma}_{\varphi(0)},\sigma} = W_{K^{\gamma}_{\sigma(0)}, \varphi}^{*}.
\end{equation}
We point out that this formula is in fact equivalent to the formula of $C_{\varphi}^{*}$ given by Cowen and MacCluer in \cite[Theorem 16]{CowenSzeged2000}, which can be written as
\begin{equation*}
C_{\varphi}^{*} = M_{K^{\gamma}_{\varphi(0)}}C_{\sigma}M^{*}_{1/K^{\gamma}_{\sigma(0)}}.
\end{equation*}
Here for an analytic function $g$ on the unit ball, $M_{g}$ denotes the operator of multiplication by $g$ on $H_{\gamma}$.

For any point $p$ in $\mathbb{B}_n$, it follows from \cite[Definition~2.2.1]{RudinSpringer1980} that the involution $\varphi_p$ of $\mathbb{B}_n$ has the form $\varphi_p(z)=\frac{Tz+p}{1-\langle z,p\rangle}$ for some self-adjoint operator $T$ depending on $p$. This implies that the adjoint of $\varphi_p$ is the same as $\varphi_p$. Now let $f,\varphi$ satisfy \eqref{Eqn:normal_WCOs} in Theorem \ref{T:normal_WCOs}. Then the adjoint $\sigma$ of $\varphi$ has the form $\sigma(z)=\varphi_p(A^{*}\varphi_p(z))$ for $z\in\mathbb{B}_n$. (Note that the adjoint of $\psi_1\circ\psi_2$ is the composition of the adjoint of $\psi_2$ and the adjoint of $\psi_1$, in this order, see \cite[Lemma 12]{CowenSzeged2000}.) In particular, $\sigma(p)=\varphi_p(A^{*}\varphi_p(p))=p$. We thus obtain
\begin{equation*}
f = \alpha\frac{k^{\gamma}_p}{k^{\gamma}_p\circ\varphi} = \alpha\frac{K^{\gamma}_p}{K^{\gamma}_p\circ\varphi} = \alpha\frac{K^{\gamma}_{\sigma(p)}}{K^{\gamma}_{p}\circ\varphi} = \frac{\alpha}{\overline{K}^{\gamma}_{\varphi(0)}(p)}K^{\gamma}_{\sigma(0)}.
\end{equation*}
The last equality follows from the second identity in Lemma \ref{L:kernel_identity_adjoint}. Therefore we see that $f$ is a constant multiple of $K^{\gamma}_{\sigma(0)}$.

In the rest of this section, we assume that $\varphi$ is a linear fractional map and $f=K^{\gamma}_{\sigma(0)}$, where as above $\sigma$ is the adjoint map of $\varphi$. We look for conditions for which the weighted composition operator $W_{f,\varphi}$ is normal. We emphasize here that in the case $\varphi$ has a fixed point $p$ in $\mathbb{B}_n$, Theorem \ref{T:normal_WCOs} provides a complete answer: $W_{f,\varphi}$ is normal if and only if $\varphi(z)=\varphi_p(A\varphi_p(z))$ for some normal operator $A$ on $\mathbb{B}_n$. The result below does not require that $\varphi$ have a fixed point in $\mathbb{B}_n$.

\begin{proposition}\label{P:normal_fractional_WCOs}
Suppose $\varphi$ is a linear fractional self-map of $\mathbb{B}_n$ and $\sigma$ is its adjoint. Let $\gamma>0$ and put $f=K^{\gamma}_{\sigma(0)}$. Then the operator $W_{f,\varphi}$ is normal on $H_{\gamma}$ if and only if $|\varphi(0)|=|\sigma(0)|$ and $\varphi\circ\sigma = \sigma\circ\varphi$.
\end{proposition}
\begin{proof}
Using \eqref{Eqn:adjusted_Cowen_MacCluer} and \eqref{Eqn:multiplication_WCOs}, we compute
\begin{align*}
W_{f,\varphi}^{*}W_{f,\varphi} & = W_{K^{\gamma}_{\varphi(0)},\sigma}W_{K^{\gamma}_{\sigma(0)},\varphi} = W_{K^{\gamma}_{\varphi(0)}\cdot K^{\gamma}_{\sigma(0)}\circ\sigma,\,\varphi\circ\sigma},\\
W_{f,\varphi}W_{f,\varphi}^{*} & = W_{K^{\gamma}_{\sigma(0)},\varphi}W_{K^{\gamma}_{\varphi(0)},\sigma} = W_{K^{\gamma}_{\sigma(0)}\cdot K^{\gamma}_{\varphi(0)}\circ\varphi,\,\sigma\circ\varphi}.
\end{align*}
This shows that $W_{f,\varphi}$ is normal if and only if $\varphi\circ\sigma=\sigma\circ\varphi$ and
\begin{equation}\label{Eqn:necessary_normal_WCOs}
 K^{\gamma}_{\varphi(0)}\cdot K^{\gamma}_{\sigma(0)}\circ\sigma=K^{\gamma}_{\sigma(0)}\cdot K^{\gamma}_{\varphi(0)}\circ\varphi.
\end{equation}

By the first identity in Lemma \ref{L:kernel_identity_adjoint}, the left hand side of \eqref{Eqn:necessary_normal_WCOs} equals $$\overline{K}^{\gamma}_{\sigma(0)}(\sigma(0))K^{\gamma}_{\varphi(\sigma(0))} = (1-|\sigma(0)|^2)^{-\gamma}K^{\gamma}_{\varphi(\sigma(0))}.$$ Similarly, by the second identity in Lemma \ref{L:kernel_identity_adjoint}, the right hand side of \eqref{Eqn:necessary_normal_WCOs} equals $$\overline{K}^{\gamma}_{\varphi(0)}(\varphi(0))K^{\gamma}_{\sigma(\varphi(0))}=(1-|\varphi(0)|^2)^{-\gamma}K^{\gamma}_{\sigma(\varphi(0))}.$$ Thus \eqref{Eqn:necessary_normal_WCOs} holds if and only if $|\sigma(0)|=|\varphi(0)|$ and $\varphi(\sigma(0))=\sigma(\varphi(0))$. The latter is certainly true if $\varphi\circ\sigma=\sigma\circ\varphi$. 

Therefore, the operator $W_{K^{\gamma}_{\sigma(0)},\varphi}$ is normal if and only if $\varphi\circ\sigma = \sigma\circ\varphi$ and $|\varphi(0)|=|\sigma(0)|$, which completes the proof of the proposition.
\end{proof}

\begin{remark} Proposition \ref{P:normal_fractional_WCOs} in the case of the Hardy space on the unit disk ($n=1$) was obtained by Bourdon and Narayan in \cite[Proposition 12]{BourdonJMAA2010} but their conclusion was stated in a slightly different way.
\end{remark}

\begin{remark}\label{R:sufficient_normal_WCOs} 
In the case $n=1$ and $\varphi(z)=\frac{az+b}{cz+d}$ for complex numbers $a,b,c,d$, an easy calculation shows that the conditions obtained in Proposition \ref{P:normal_fractional_WCOs} are equivalent to $|b|=|c|$ and $\overline{a}b-\overline{c}d = b\overline{d}-a\overline{c}$.

When $n\geq 2$ and $\varphi(z)=\frac{Az+B}{\langle z,C\rangle +d}$, the conditions can then be expressed in terms of $A, B, C$ and $d$. We leave this to the interested reader. 
\end{remark}

We conclude the section by a result taken from \cite[Proposition 13]{BourdonJMAA2010} with a slightly modified proof using Remark \ref{R:sufficient_normal_WCOs}.

\begin{proposition} Suppose that $\varphi$  is a linear fractional self-map of the unit disk of parabolic type (so there is an $\omega$  with $|\omega|=1$ such that $\varphi(\omega)=\omega$ and $\varphi'(\omega)=1$). Then the operator $W_{K^{\gamma}_{\sigma(0)},\varphi}$ is normal on $H_{\gamma}$ for any $\gamma>0$. Here as before, $\sigma$ is the adjoint map of $\varphi$.
\end{proposition}

\begin{proof} As it is explained in the proof of \cite[Proposition 13]{BourdonJMAA2010}, we only need to consider $\omega=1$ and $\varphi$ of the form $$\varphi(z)=\dfrac{(2-t)z+t}{-tz+(2+t)}, \quad\text{ where } {\rm Re}(t)\geq 0.$$
Since $a=2-t, b=t, c=-t$ and $d=2+t$, we have $|b|=|c|$ and $\overline{a}b-\overline{c}d = b\overline{d}-a\overline{c}=4{\rm Re}(t)$. The conclusion now follows from Remark \ref{R:sufficient_normal_WCOs} and Proposition \ref{P:normal_fractional_WCOs}.
\end{proof}

\section{Self-adjoint weighted composition operators}

In this section we characterize when the adjoint of a weighted composition operator on $H_{\gamma}$ is another weighted composition operator. As a consequence, we determine necessary and sufficient conditions for which the operator $W_{f,\varphi}$ is a self-adjoint operator. This generalizes the characterizations obtained in \cite{CowenTAMS2010,CowenPreprint2010}, where the one-dimensional case is considered. Furthermore, our solutions to the equation $W^{*}_{g,\psi}=W_{f,\varphi}$ seems to be new even in one dimension.

We will need the following elementary result regarding maps on the unit ball of a Hilbert space. The existence of the linear extensions follows from a similar argument as in the proof of Lemma \ref{L:unitary_on_Cn}. The boundedness is well known and it is a consequence of the closed graph theorem.

\begin{lemma}\label{L:adjoint_operators} Let $\mathcal{M}$ be Hilbert space with an inner product $\langle,\rangle$. Suppose $F$ and $G$ are two maps from the unit ball $\mathcal{B}$ of $\mathcal{M}$ into $\mathcal{M}$ such that for all $z,w\in\mathcal{B}$, $\langle F(w),z\rangle = \langle w,G(z)\rangle$. Then there is a bounded linear operator $A$ on $\mathcal{M}$ such that $F(w)=Aw$ and $G(w)=A^{*}w$ for all $w\in\mathcal{B}$.
\end{lemma}

By \eqref{Eqn:adjusted_Cowen_MacCluer} we see that the adjoint of $W_{K^{\gamma}_{\sigma(0)},\varphi}$ is the weighted composition operator $W_{K^{\gamma}_{\varphi(0)},\sigma}$ when $\varphi$ is a linear fractional map and $\sigma$ is the adjoint map of $\varphi$. Our main result in this section shows that any non-zero weighted composition operator whose adjoint is a weighted composition operator must be a constant multiple of an operator of this form.

\begin{theorem}\label{T:adjoint_WCOs}
Let $f,g$ be analytic functions on $\mathbb{B}_n$ and $\varphi,\psi$ be analytic self-maps of $\mathbb{B}_n$. Then $W_{f,\varphi}$ and $W_{g,\psi}$ are non-zero bounded operators on $H_{\gamma}$ and $W^{*}_{g,\psi} = W_{f,\varphi}$ if and only if there are vectors $c,d$ in $\mathbb{B}_n$, a linear operator $A$ on $\mathbb{C}^n$ and a non-zero complex number $\alpha$ such that
\begin{equation}\label{Eqn:equality_adjoint_functions}
\varphi(z)=\dfrac{d+Az}{1-\langle z,c\rangle} \text{ and } \psi(z)=\dfrac{c+A^{*}z}{1-\langle z,d\rangle} \text{ for all } z\in\mathbb{B}_n,
\end{equation}
and $f=\alpha\, K^{\gamma}_{c} = \alpha\, K^{\gamma}_{\psi(0)},\; g=\overline{\alpha}\, K^{\gamma}_{d} = \overline{\alpha}\, K^{\gamma}_{\varphi(0)}$. In particular, the maps $\varphi$ and $\psi$ are linear fractional maps.
\end{theorem}

\begin{remark}
Note that the map $\psi$ in \eqref{Eqn:equality_adjoint_functions} is the adjoint of $\varphi$. Thus Theorem \ref{T:adjoint_WCOs} says, in particular, that if $W_{g,\psi}$ is the adjoint operator of $W_{f,\varphi}$, then $\psi$ is the adjoint of $\varphi$.
\end{remark}

\begin{proof}
Suppose first $W_{g,\psi}^{*}=W_{f,\varphi}$ on $H_{\gamma}$ and they are non-zero operators. For any $z$ and $w$ in $\mathbb{B}_n$, using \eqref{Eqn:adjoint_WCO} we have
\begin{align}\label{Eqn:equality_adjoint}
f(w)K^{\gamma}_{z}(\varphi(w)) = (W_{f,\varphi}K^{\gamma}_z)(w) = (W^{*}_{g,\psi}K^{\gamma}_z)(w) = \overline{g}(z)K^{\gamma}_{\psi(z)}(w).
\end{align}

Letting $z=0$ in \eqref{Eqn:equality_adjoint} gives $f(w)=\overline{g}(0)K^{\gamma}_{\psi(0)}(w)=\overline{g}(0)\big(1-\langle w,\psi(0)\rangle\big)^{-\gamma}$ for $w\in\mathbb{B}_n$. This, in particular, implies $f(0)=\overline{g}(0)$, which is non-zero by the assumption that operators are non-zero.

Letting $w=0$ in \eqref{Eqn:equality_adjoint} gives 
\begin{equation*}
\overline{g}(z)=f(0)K^{\gamma}_{z}(\varphi(0))=f(0)\big(1-\langle \varphi(0),z\rangle\big)^{-\gamma} \text{ for } z\in\mathbb{B}_n.
\end{equation*}

Substituting the formulas for $f, g$ and $K(\cdot,\cdot)$ into \eqref{Eqn:equality_adjoint} and canceling the constants, we obtain
\begin{align*}
(1-\langle w,\psi(0)\rangle)^{-\gamma}(1-\langle \varphi(w),z\rangle)^{-\gamma} = (1-\langle\varphi(0),z\rangle)^{-\gamma}(1-\langle w,\psi(z)\rangle)^{-\gamma}.
\end{align*}
This identity implies 
\begin{equation}\label{Eqn:equality_adjoint_maps}
(1-\langle w,\psi(0)\rangle)(1-\langle \varphi(w),z\rangle) = (1-\langle\varphi(0),z\rangle)(1-\langle w,\psi(z)\rangle).
\end{equation}
An easy calculation then gives
\begin{equation*}
\Big\langle \big(1-\langle w,\psi(0)\rangle\big)\varphi(w)-\varphi(0),z\Big\rangle = \Big\langle w, \big(1-\langle z,\varphi(0)\rangle\big)\psi(z)-\psi(0)\Big\rangle.
\end{equation*}
Using Lemma \ref{L:adjoint_operators}, we conclude that there exists a linear operator $A$ on $\mathbb{C}^n$ such that
\begin{equation*}
\varphi(w)=\dfrac{\varphi(0)+Aw}{1-\langle w,\psi(0)\rangle} \text{ and } \psi(z)=\dfrac{\psi(0)+A^{*}z}{1-\langle z,\varphi(0)\rangle} \text{ for all } w,z\in\mathbb{B}_n.
\end{equation*}
Put $\alpha=f(0), c=\psi(0)$ and $d=\varphi(0)$, we see that $f,g$ and $\varphi,\psi$ satisfy \eqref{Eqn:equality_adjoint_functions}.

For the converse, suppose $f,g$ and $\varphi,\psi$ are as above such that $\varphi$ and $\psi$ map the unit ball into itself. Since $W_{f,\varphi} = \alpha W_{K^{\gamma}_{\psi(0)},\varphi}$ and $W_{g,\psi} = \overline{\alpha} W_{K^{\gamma}_{\varphi(0)},\psi}$, \eqref{Eqn:adjusted_Cowen_MacCluer} gives $W_{g,\psi}=W^{*}_{f,\varphi}$ on $H_{\gamma}$, which is equivalent to $W_{f,\varphi} = W_{g,\psi}^{*}$. The boundedness of these operators on $H_{\gamma}$ follows from Remark \ref{R:boundedness_WCOs}.
\end{proof}

As an immediate application of Theorem \ref{T:adjoint_WCOs}, we obtain a characterization of self-adjoint weighted composition operators.

\begin{corollary}\label{C:self-adjoint_WCOs}
Let $f$ be an analytic function and $\varphi$ an analytic self-map of $\mathbb{B}_n$. Then $W_{f,\varphi}$ is a non-zero self-adjoint bounded operator on $H_{\gamma}$ if and only if there is a vector $c\in\mathbb{B}_n$, a self-adjoint linear operator $A$ on $\mathbb{C}^n$ and a real number $\alpha$ such that $f = \alpha\,K^{\gamma}_{c}=\alpha\, K^{\gamma}_{\varphi(0)}$ and $\varphi(z)=\frac{c+Az}{1-\langle z,c\rangle}$ for $z\in\mathbb{B}_n$.
\end{corollary}
\begin{proof}
Since $W_{f,\varphi}^{*}=W_{f,\varphi}$, Theorem \ref{T:adjoint_WCOs} shows that there are vectors $c,d$ in $\mathbb{B}_n$, a linear operator $A$ on $\mathbb{C}^n$ and a complex number $\alpha$ such that for all $z\in\mathbb{B}_n$, $f(z)=\alpha\,K^{\gamma}_{d} = \overline{\alpha}\,K^{\gamma}_{c}$ and $\varphi(z)=\frac{d+Az}{1-\langle z,c\rangle} = \frac{c+A^{*}z}{1-\langle z,d\rangle}$. This shows that $\overline{\alpha}=\alpha$, $c=d$ and $A^{*}=A$ and hence $f,\varphi$ have the required form.
\end{proof}

In \cite{CowenTAMS2010, CowenPreprint2010}, the authors go on to describe the eigenvectors, eigenvalues and other spectral properties of self-adjoint weighted composition operators on $H_{\gamma}$ ($\gamma\geq 1$) of the unit disk. Their analysis is based on the classification of linear fractional self-maps of the unit disk.

In dimension $n\geq 2$ and in the case $\varphi$ has a fixed point in $\mathbb{B}_n$ (the elliptic case), eigenvectors, eigenvalues and the spectrum of the self-adjoint operator $W_{f,\varphi}$ can be described as in Proposition \ref{P:spectrum_normal_WCOs} and in the discussion preceding this proposition. The cases where all the fixed points of $\varphi$ lie on the unit sphere (the parabolic and hyperbolic cases) are, we believe, more complicated and seem to require more careful analysis. We leave this open for future research.

\end{document}